\documentclass[sn-mathphys]{sn-jnl}

\jyear{2023}%
\theoremstyle{thmstyleone}%
\newtheorem{theorem}{Theorem}
\newtheorem{lemma}{Lemma}
\begin{document}
\title[The behaviour of the Gauss-Radau upper bound]{The behaviour of the Gauss-Radau upper bound of~the~error norm in CG}

\author[1]{\fnm{G\'erard} \sur{Meurant}}\email{gerard.meurant@gmail.com}
\author*[2]{\fnm{Petr} \sur{Tich\'y}}\email{petr.tichy@mff.cuni.cz}
\affil[1]{\orgaddress{\city{Paris},  \country{France}}}
\affil*[2]{\orgdiv{Faculty of Mathematics and Physics}, \orgname{Charles University}, \orgaddress{\street{Sokolovsk\'a 83}, \city{Prague}, \postcode{18675}, \country{Czech Republic}}}

\abstract{
Consider the problem of solving systems of linear algebraic equations $Ax=b$ with a real symmetric positive definite matrix $A$ using the conjugate gradient (CG) method.
To stop the algorithm at the appropriate moment,
it is important to monitor the quality
of the approximate solution.
One of the most relevant quantities
for measuring the quality of the approximate solution is the $A$-norm of the error. This quantity cannot be easily {computed}, however, it can be estimated.
In this paper we discuss and analyze the behaviour
of the Gauss-Radau upper bound on the $A$-norm of the error, based
on viewing CG as a procedure for approximating a certain Riemann-Stieltjes integral.
This upper bound depends on a prescribed underestimate $\mu$ to the smallest eigenvalue of $A$.
We concentrate on explaining a phenomenon observed during computations
showing that, in later CG iterations, the upper bound loses its accuracy,
and is almost independent of~$\mu$.
We construct a model problem that is used to demonstrate and study the behaviour of the upper bound in dependence of~$\mu$, and developed
formulas that are helpful in understanding this behavior.
We show that the above mentioned phenomenon is closely related
to the convergence of the smallest Ritz value to the smallest eigenvalue
of $A$. It occurs when the smallest Ritz value is a better approximation
to the smallest eigenvalue than the prescribed underestimate $\mu$.
%
{We also suggest an adaptive strategy for improving the
accuracy of the upper bounds in the
previous iterations.}

}



\maketitle

\section{Introduction}

Our aim in this paper is to explain the origin of the problems that
have been noticed \cite{MeTi2019} when {computing} Gauss-Radau quadrature
upper bounds of the $A$-norm of the error in the Conjugate Gradient
(CG) algorithm for solving linear systems $Ax=b$ with a symmetric
positive definite matrix of order $N$.


The connection between CG and Gauss quadrature has been
known since the seminal paper of Hestenes and Stiefel \cite{HeSt1952}
in 1952. This link has been exploited by Gene H.~Golub and his collaborators
to bound or estimate the $A$-norm of the error in CG during the iterations;
see
\cite{DaEiGo1972,DaGoNa1979,GoMe1994,GoMe1997,GoSt1994,Me1997,Me1999,MeTi2013,MeTi2014,MeTi2019,StTi2002,StTi2005}.

Using a fixed node $\mu$ smaller than the smallest eigenvalue of
$A$ and the Gauss-Radau quadrature rule, an upper bound for the $A$-norm
of the error can be easily computed. Note that it is useful to have
an upper bound of the error norm to stop the CG iterations. In theory,
the closer $\mu$ is to the smallest eigenvalue, the closer is the
bound to the norm.
%
{Concerning the approximation properties of the upper bound,
we observed in many examples that
in earlier iterations,
the bound is approximating the $A$-norm of the error quite well,
and that the quality of approximation is improving with increasing iterations. However,
in later CG iterations, the bound suddenly becomes worse:
it is delayed, almost independent of $\mu$,
and does not represent a good approximation to the $A$-norm of the error any more.
Such a behavior of the upper bound can be observed also in
exact arithmetic.
Therefore, the problem of the loss of accuracy of the upper bound in later iterations is not directly} linked to rounding
errors and has to be explained.


The Gauss quadrature bounds of the error norm were obtained by using
the connection of CG to the Lanczos algorithm and modifications of
the tridiagonal matrix which is generated by this algorithm and implicitly
by CG. This is why we start in Section~\ref{sec:lanczos} with the Lanczos algorithm.
In Section~\ref{sec:estimation} we discuss the relation with CG and how the Gauss-Radau
upper bound is computed. A~model problem showing the problems arising
with the Gauss-Radau bound in ``exact'' arithmetic is constructed in Section~\ref{sec:model}.
In Sections~\ref{sec:analysis} to \ref{sec:detection} we give an analysis that explains that the problems start when
the distance of the smallest Ritz value to the smallest eigenvalue becomes smaller than
the distance of~$\mu$ to the smallest eigenvalue. We also explain why the Gauss-Radau upper bound
becomes almost independent of~$\mu$.
{In Section~\ref{sec:improved} we present an algorithm for
improving the upper bounds in previous CG iterations
such that the relative accuracy of the upper bounds is
guaranteed to be smaller than a prescribed tolerance.}
Conclusions are given in Section~\ref{sec:conclusions}.

\section{The Lanczos algorithm}
\label{sec:lanczos}
Given a starting vector $v$ and a symmetric matrix $A\in\mathbb{R}^{N\times N}$,
one can consider a sequence of nested Krylov subspaces
\[
\mathcal{K}_{k}(A,v)\equiv\mathrm{span}\{v,Av,\dots,A^{k-1}v\},\qquad k=1,2,\dots
\]
The dimension of these subspaces can increase up to an index $n$
called the \emph{grade of $v$ with respect to $A$}, at which the
maximal dimension is attained, and $\mathcal{K}_{n}(A,v)$ is invariant
under multiplication with $A$.
\begin{algorithm}[h]
\caption{Lanczos algorithm}
\label{alg:lanczos}

\begin{algorithmic}[0]

\State \textbf{input} $A$, $v$

\State $\beta_{0}=0$, $v_{0}=0$

\State $v_{1}=v/\|v\|$

\For{$k=1,\dots$}

\State $w=Av_{k}-\beta_{k-1}v_{k-1}$

\State $\alpha_{k}=v_{k}^{T}w$

\State $w=w-\alpha_{k}v_{k}$

\State $\beta_{k}=\|w\|$

\State $v_{k+1}=w/\beta_{k}$

\EndFor

\end{algorithmic}
\end{algorithm}

Assuming that $k<n$, the Lanczos algorithm (Algorithm \ref{alg:lanczos})
computes an orthonormal basis $v_{1},\dots,v_{k+1}$ of the Krylov
subspace $\mathcal{K}_{k+1}(A,v)$. The basis vectors $v_{j}$ satisfy
the matrix relation
\[
AV_{k}=V_{k}T_{k}+\beta_{k}v_{k+1}e_{k}^{T}
\]
where $e_{k}$ is the last column of the identity matrix of order
$k$, $V_{k}=[v_{1}\cdots v_{k}]$ and $T_{k}$ is the $k\times k$
symmetric tridiagonal matrix of the recurrence coefficients computed
in Algorithm \ref{alg:lanczos}:
\[
T_{k}=\left[\begin{array}{cccc}
\alpha_{1} & \beta_{1}\\
\beta_{1} & \ddots & \ddots\\
 & \ddots & \ddots & \beta_{k-1}\\
 &  & \beta_{k-1} & \alpha_{k}
\end{array}\right].
\]
The coefficients $\beta_{j}$ being positive, $T_{k}$ is a so-called
Jacobi matrix. If $A$ is positive definite, then $T_{k}$ is positive
definite as well. In the following we will assume for simplicity that
the eigenvalues of $A$ are simple and sorted such that
\[
\lambda_{1}<\lambda_{2}<\cdots<\lambda_{N}.
\]

\subsection{Approximation of the eigenvalues}
\label{sec:eigenvalues}
The eigenvalues of $T_{k}$ (Ritz values) are usually used as approximations
to the eigenvalues of $A$. The quality of the approximation can be
measured using $\beta_{k}$ and the last components of the normalized
eigenvectors of $T_{k}$. In more detail, consider the spectral decomposition
of $T_{k}$ in the form
\[
T_{k}=S_{k}\Theta_{k}S_{k}^{T},\quad\Theta_{k}=\mathrm{diag}\left(\theta_{1}^{(k)},\dots,\theta_{k}^{(k)}\right),\quad S_{k}^{T}S_{k}=S_{k}S_{k}^{T}=I_{k},
\]
where $I_k$ is the identity matrix of order $k$, and assume that the Ritz values are sorted such that
\[
\theta_{1}^{(k)}<\theta_{2}^{(k)}<\cdots<\theta_{k}^{(k)}.
\]
Denote $s_{i,j}^{(k)}$ the entries and $s_{:,j}^{(k)}$ the columns
of $S_{k}$. Then it holds that
\begin{equation}
\min_{i=1,\dots,N}\lvert\lambda_{i}-\theta_{j}^{(k)}\rvert\leq\left\Vert A\left(V_{k}s_{:,j}^{(k)}\right)-\theta_{j}^{(k)}\left(V_{k}s_{:,j}^{(k)}\right)\right\Vert =\beta_{k}\lvert s_{k,j}^{(k)}\rvert,\label{eq:simplebound}
\end{equation}
$j=1,\dots,k$. Since the Ritz values $\theta_{j}^{(k)}$ can be seen
as Rayleigh quotients, one can improve the bound \eqref{eq:simplebound}
using the gap theorem; see \cite[p.~244]{B:Pa1998} or \cite[p.~206]{B:De1997}.
In particular, let $\lambda_{\ell}$ be an eigenvalue of $A$ closest
to $\theta_{j}^{(k)}$. Then
\[
\lvert\lambda_{\ell}-\theta_{j}^{(k)}\rvert\leq\frac{\left(\beta_{k}s_{k,j}^{(k)}\right)^{2}}{\mathrm{gap}_{j}^{(k)}},\qquad\mathrm{gap}_{j}^{(k)}=\min_{i\neq\ell}\lvert\lambda_{i}-\theta_{j}^{(k)}\rvert.
\]

In the following we will be interested in the situation when the smallest
Ritz value $\theta_{1}^{(k)}$ closely approximates the smallest eigenvalue
of $A$. If $\lambda_{1}$ is the eigenvalue of $A$ closest to $\theta_{1}^{(k)}>\lambda_{1}$,
then using the gap theorem and \cite[Corollary 11.7.1 on p.~246]{B:Pa1998},
\begin{equation}
\frac{\left(\beta_{k}s_{k,1}^{(k)}\right)^{2}}{\lambda_{n}-\lambda_{1}}\leq\theta_{1}^{(k)}-\lambda_{1}\leq\frac{\left(\beta_{k}s_{k,1}^{(k)}\right)^{2}}{\lambda_{2}-\theta_{1}^{(k)}},\label{eq:gap}
\end{equation}
giving the bounds
\begin{equation}
\lambda_{2}-\theta_{1}^{(k)}\leq\frac{\left(\beta_{k}s_{k,1}^{(k)}\right)^{2}}{\theta_{1}^{(k)}-\lambda_{1}}\leq\lambda_{n}-\lambda_{1}.\label{eq:bound1}
\end{equation}

It is known (see, for instance, \cite{B:Me2006}) that the squares
of the last components of the eigenvectors are given by
\[
\left(s_{k,j}^{(k)}\right)^{2}=\left\lvert\frac{\chi_{1,k-1}(\theta_{j}^{(k)})}{\chi_{1,k}^{'}(\theta_{j}^{(k)})}\right\rvert,
\]
where $\chi_{1,\ell}$ is the characteristic polynomial of $T_{\ell}$
and $\chi_{1,\ell}^{'}$ denotes its derivative, i.e.,
\[
\left(s_{k,j}^{(k)}\right)^{2}=\frac{\theta_{j}^{(k)}-\theta_{1}^{(k-1)}}{\theta_{j}^{(k)}-\theta_{1}^{(k)}}\cdots\frac{\theta_{j}^{(k)}-\theta_{j-1}^{(k-1)}}{\theta_{j}^{(k)}-\theta_{j-1}^{(k)}}\frac{\theta_{j}^{(k-1)}-\theta_{j}^{(k)}}{\theta_{j+1}^{(k)}-\theta_{j}^{(k)}}\cdots\frac{\theta_{k-1}^{(k-1)}-\theta_{j}^{[k)}}{\theta_{k}^{(k)}-\theta_{j}^{(k)}}.
\]
The right-hand side is positive due to the interlacing property of
the Ritz values for symmetric tridiagonal matrices. In particular,
\begin{equation}
\left(s_{k,1}^{(k)}\right)^{2}=\frac{\theta_{1}^{(k-1)}-\theta_{1}^{(k)}}{\theta_{2}^{(k)}-\theta_{1}^{(k)}}\cdots\frac{\theta_{k-1}^{(k-1)}-\theta_{1}^{(k)}}{\theta_{k}^{(k)}-\theta_{1}^{(k)}}.\label{eq:s1k}
\end{equation}
When the smallest Ritz value $\theta_{1}^{(k)}$ converges to $\lambda_{1}$,
this last component squared converges to zero; see also \eqref{eq:bound1}.

\subsection{Modification of the tridiagonal matrix}

\label{sec:modified}
{Given $\mu < \theta_{1}^{(k)}$, let us consider the
problem of finding the coefficient $\alpha_{k+1}^{(\mu)}$ such that
the modified matrix}
%
\begin{equation}
T_{k+1}^{(\mu)}=\left[\begin{array}{ccccc}
\alpha_{1} & \beta_{1}\\
\beta_{1} & \ddots & \ddots\\
 & \ddots & \ddots & \beta_{k-1}\\
 &  & \beta_{k-1} & \alpha_{k} & \beta_{k}\\
 &  &  & \beta_{k} & \alpha_{k+1}^{(\mu)}
\end{array}\right]\label{eq:Tmu}
\end{equation}
has the prescribed $\mu$ as an eigenvalue.
{The connection of this
problem to Gauss-Radau quadrature rule will be explained in
Section~\ref{sec:estimation}.}
In \cite[pp.~331-334]{Go1973}
it has been shown that at iteration $k+1$
\[
\alpha_{k+1}^{(\mu)}\,=\,\mu+\zeta_{k}^{(\mu)}
\]
where $\zeta_{k}^{(\mu)}$ is the last component of the vector $y$,
{solution of} the linear system
\begin{equation}
(T_{k}-\mu I)y=\beta_{k}^{2}e_{k}.\label{eq-Tkmu-1}
\end{equation}
From \cite[Section 3.4]{MeTi2013}, the modified coefficients $\alpha_{k+1}^{(\mu)}$
can be computed recursively using
\begin{equation}
\alpha_{j+1}^{(\mu)}=\mu+\frac{\beta_{j}^{2}}{\alpha_{j}-\alpha_{j}^{(\mu)}},\qquad\alpha_{1}^{(\mu)}=\mu,\qquad j=1,\dots,k.\label{eq:alphaupdate}
\end{equation}
{Using} the spectral factorization of $T_{k}$, we can now prove
the following lemma.


\begin{lemma}
\label{lem:alphamu} Let $\mu<\theta_{1}^{(k)}$. Then it holds that
\begin{equation}
\alpha_{k+1}^{(\mu)}\,=\,\mu+\sum_{i=1}^{k}\eta_{i,k}^{(\mu)},\qquad\eta_{i,k}^{(\mu)}\equiv\frac{\left(\beta_{k}s_{k,i}^{(k)}\right)^{2}}{\theta_{i}^{(k)}-\mu}.\label{eq:alphaspectral}
\end{equation}
If $\mu<\lambda<\theta_{1}^{(k)}$, then $\alpha_{k+1}^{(\mu)}<\alpha_{k+1}^{(\lambda)}$.
Consequently, if $\mu<\theta_{1}^{(k+1)},$ then $\alpha_{k+1}^{(\mu)}<\alpha_{k+1}$.
\end{lemma}


\begin{proof}
Since $\mu<\theta_{1}^{(k)}$ the matrix $T_{k}-\mu I$ in \eqref{eq-Tkmu-1}
is positive definite and, therefore, nonsingular. Hence,
\begin{equation}
\zeta_{k}^{(\mu)}=e_{k}^{T}y=\beta_{k}^{2}e_{k}^{T}(T_{k}-\mu I)^{-1}e_{k}=\sum_{i=1}^{k}\frac{\left(\beta_{k}s_{k,i}^{(k)}\right)^{2}}{\theta_{i}^{(k)}-\mu}\label{eq:zeta}
\end{equation}
so that \eqref{eq:alphaspectral} holds. From \eqref{eq:alphaspectral} it is obvious that if $\mu<\lambda<\theta_{1}^{(k)}$,
then $\alpha_{k+1}^{(\mu)}<\alpha_{k+1}^{(\lambda)}$.

Finally, taking
$\lambda=\theta_{1}^{(k+1)}<\theta_{1}^{(k)}$ (because of the interlacing
of the Ritz values) {we obtain $\alpha_{k+1}^{(\lambda)}=\alpha_{k+1}$ by construction.}
\end{proof}

\section{{CG} and error norm estimation}
\label{sec:estimation}

When solving a linear system $Ax=b$ with a symmetric and positive
definite matrix $A$, the CG method (Algorithm \ref{alg:cg})
\begin{algorithm}[h]
\caption{Conjugate gradient algorithm}
\label{alg:cg}

\begin{algorithmic}[0]

\State \textbf{input} $A$, $b$, $x_{0}$

\State $r_{0}=b-Ax_{0}$

\State $p_{0}=r_{0}$

\For{$k=1,\dots$ until convergence}

\State $\gamma_{k-1}=\frac{r_{k-1}^{T}r_{k-1}}{p_{k-1}^{T}Ap_{k-1}}$

\State $x_{k}=x_{k-1}+\gamma_{k-1}p_{k-1}$

\State $r_{k}=r_{k-1}-\gamma_{k-1}Ap_{k-1}$\hspace*{5mm}
\smash{$\left.\begin{array}{@{}c@{}}\\{}\\{}\\{}\\{}\\{}\end{array}         \right\} \begin{tabular}{l}{\tt cgiter(k-1)}\end{tabular}$}

\State $\delta_{k}=\frac{r_{k}^{T}r_{k}}{r_{k-1}^{T}r_{k-1}}$

\State $p_{k}=r_{k}+\delta_{k}p_{k-1}$

\EndFor

\end{algorithmic}
\end{algorithm}
is the method of choice. In exact arithmetic, the CG iterates $x_{k}$
minimize the $A$-norm of the error over the manifold $x_{0}+\mathcal{K}_{k}(A,r_{0})$,
\[
\|x-x_{k}\|_{A}=\min_{y\in x_{0}+\mathcal{K}_{k}(A,r_{0})}\|x-y\|_{A},
\]
and the residual vectors $r_{k}=b-Ax_{k}$ are proportional to the
Lanczos vectors $v_{j}$,
\[
v_{j+1}=(-1)^{j}\frac{r_{j}}{\|r_{j}\|}\,,\qquad j=0,\dots,k.
\]
Thanks to this close relationship between the CG and Lanczos algorithms,
it can be shown (see, for instance \cite{B:Me2006}) that the recurrence
coefficients computed in both algorithms are connected via $\alpha_{1}=\gamma_{0}^{-1}$
and
\begin{equation}
\beta_{j}=\frac{\sqrt{\delta_{j}}}{\gamma_{j-1}},\quad\alpha_{j+1}=\frac{1}{\gamma_{j}}+\frac{\delta_{j}}{\gamma_{j-1}},\quad j=1,\dots,k-1.\label{eq:CGLanczos}
\end{equation}
Writing \eqref{eq:CGLanczos} in matrix form, we find out that CG
computes implicitly the $LDL^{T}$ factorization $T_{k}=L_{k}D_{k}L_{k}^{T}$,
where
\begin{equation}
L_{k}=\left[\begin{array}{cccc}
1\\
\sqrt{\delta_{1}} & \ddots\\
 & \ddots & \ddots\\
 &  & \sqrt{\delta_{k-1}} & 1
\end{array}\right],\quad D_{k}=\left[\begin{array}{cccc}
\gamma_{0}^{-1}\\
 & \ddots\\
 &  & \ddots\\
 &  &  & \gamma_{k-1}^{-1}
\end{array}\right].\label{eq:matrixLDL}
\end{equation}
Hence the matrix $T_{k}$ is known implicitly in CG.

\subsection{Modification of the factorization of $T_{k+1}$}

Similarly as in Section~\ref{sec:modified} we can ask how to modify
the Cholesky factorization of $T_{k+1}$, that is available in CG,
such that the resulting matrix $T_{k+1}^{(\mu)}$ given implicitly
in factorized form has the prescribed eigenvalue $\mu$. In more detail,
we look for a coefficient $\gamma_{k}^{(\mu)}$ such that
\[
T_{k+1}^{(\mu)}=L_{k+1}\left[\begin{array}{cc}
D_{k}\\
 & \left(\gamma_{k}^{(\mu)}\right)^{-1}
\end{array}\right]L_{k+1}^{T}.
\]
This problem was solved in \cite{MeTi2013} leading to an updating
formula for computing the modified coefficients
\begin{equation}
\gamma_{j+1}^{(\mu)}=\frac{\gamma_{j}^{(\mu)}-\gamma_{j}}{\mu(\gamma_{j}^{(\mu)}-\gamma_{j})+\delta_{j+1}},\ j=1,\dots,k-1,\qquad\gamma_{0}^{(\mu)}=\frac{1}{\mu}.\label{eq:gammamu}
\end{equation}
Moreover, $\gamma_{k}^{(\mu)}$ can be obtained directly from the
modified coefficient $\alpha_{k+1}^{(\mu)}$,
\begin{equation}
\gamma_{k}^{(\mu)}=\frac{1}{\alpha_{k+1}^{(\mu)}-\frac{\delta_{k}}{\gamma_{k-1}}},\label{eq:gammalpha}
\end{equation}
and vice-versa, see \cite[p.~173 and 181]{MeTi2013}.

\subsection{{Quadrature-based bounds in CG}}

{
We now briefly summarize the idea of deriving the quadrature-based bounds used in this paper. For a more detailed
description, see, e.g.,
\cite{GoMe1994,GoSt1994,GoMe1997,StTi2002,StTi2005,MeTi2013,MeTi2014}.

Let $A=Q\Lambda Q^T$ be the spectral decomposition of $A$, with
$Q=[q_1,\dots,q_N]$ orthonormal and $\Lambda=\mathrm{diag}(\lambda_1,\dots,\lambda_N)$. As we said above, for simplicity of notation, we assume that the eigenvalues of $A$
are distinct and ordered as $\lambda_1 < \lambda_2 < \dots < \lambda_N$. Let us define the weights $\omega_{i}$ by
\[
 \omega_{i}\equiv\frac{(r_{0},q_{i})^{2}}{\| r_0 \|^2}\qquad\mbox{so that}\qquad\sum_{i=1}^{N}\omega_{i}=1\,,
\]
and the (nondecreasing) stepwise constant distribution function $\omega(\lambda)$ with a finite number of points of increase
$\lambda_{1},\lambda_{2},\dots,\lambda_{N}$,
\[
 \omega(\lambda)\equiv\;\left\{ \;
 \begin{array}{ccl}
 0 & \textnormal{for} & \lambda<\lambda_{1}\,,\\[1mm]
 \sum_{j=1}^{i}\omega_{j} & \textnormal{for} & \lambda_{i}\leq\lambda<\lambda_{i+1}\,,\quad1\leq i\leq N-1\,.\\[1mm]
 1 & \textnormal{for} & \lambda_{N}\leq\lambda\,.
 \end{array}\right.\,
\]
Having the distribution function $\omega(\lambda)$ and an interval
$\langle \eta,\xi\rangle$ such that $\eta<\lambda_{1}<\lambda_{2}<\dots<\lambda_{N}<\xi$, for any continuous function $f$, one can define  the
Riemann-Stieltjes integral (see, for instance \cite{B:GoMe2010})
\[
 \int_{\eta}^{\xi}f(\lambda)\, d\omega(\lambda) = \sum_{i=1}^{N}\omega_{i}f(\lambda_{i}).
\]
For $f(\lambda)=\lambda^{-1}$, we obtain the integral representation of $\| x-x_0\|_A^2$,
\begin{eqnarray}
\label{eqn:integral}
 \int_{\eta}^{\xi}\lambda^{-1}\, d\omega(\lambda) &=& \|r_0\|^{-2}\| x-x_0\|_A^2.
\end{eqnarray}
Using the optimality of CG it can be shown that CG implicitly determines nodes and weights of the $k$-point Gauss quadrature approximation to
the Riemann-Stieltjes integral \eqref{eqn:integral}. The nodes are given by the eigenvalues of $T_k$, and the
weights by the squared first components of the normalized eigenvectors of $T_k$.
The corresponding Gauss-quadrature rule can be written in the form
\begin{equation}\label{eqn:GQk}
 \int_{\eta}^{\xi}\lambda^{-1}\, d\omega(\lambda) =
 (T_k^{-1})_{1,1} + \frac{\| x-x_k\|_A^2}{\|r_0\|^{2}},
\end{equation}
where $(T_k^{-1})_{1,1}$ represents the Gauss-quadrature approximation,
and the reminder is nothing but the scaled and squared $A$-norm of the $k$th error, i.e., the quantity of our interest.

To approximate the integral \eqref{eqn:integral},
one can also apply a modified quadrature rule.
In this paper we consider the Gauss-Radau quadrature rule
consisting in prescribing a node
$0<\mu \leq \lambda_1$ and choosing the other nodes and weights
to maximize the degree of exactness of the quadrature rule.
We can write the corresponding Gauss-Radau quadrature rule in the form
$$
 \int_{\eta}^{\xi}\lambda^{-1}\, d\omega(\lambda) =
 (({T}_k^{(\mu)})^{-1})_{1,1} + {\mathcal{R}}_k^{(\mu)},
$$
where the reminder ${\mathcal{R}}_k^{(\mu)}$ is negative, and ${T}_k^{(\mu)}$ is given by \eqref{eq:Tmu}.

The idea of deriving (basic) quadrature-based bounds in CG is to consider the  Gauss quadrature rule \eqref{eqn:GQk} at iteration $k$, and a (eventually modified) quadrature rule at iteration $k+1$,
 \begin{equation}
 \frac{\|x-x_{0}\|_{A}^{2}}{\|r_{0} \|^{2}}=\left(\widehat{T}_{k+1}^{-1}\right)_{1,1}+\widehat{\mathcal{R}}_{k+1},\label{eq:Q2}
\end{equation}
where $\widehat{T}_{k+1}=T_{k+1}$ when using the Gauss rule and $\widehat{T}_{k+1}={T}_{k+1}^{(\mu)} $ in the case of using 
the Gauss-Radau rule. From the equations \eqref{eqn:GQk} and \eqref{eq:Q2} we get
\begin{equation}
\|x-x_{k}\|_{A}^{2} = \left[
 \|r_{0} \|^{2} \left(
\left(\widehat{T}_{k+1}^{-1}\right)_{1,1}-\left(T_{k}^{-1}\right)_{1,1}\right) \right] +
\widehat{\mathcal{R}}_{k+1}.\label{eq:mainQ}
\end{equation}
The term in square brackets represents either a lower bound on $\|x-x_{k}\|_{A}^{2}$  if $\widehat{T}_{k+1}=T_{k+1}$ (because of the positive reminder),
or an upper bound if $\widehat{T}_{k+1} = {T}_{k+1}^{(\mu)}$ (because of the negative reminder). In both cases, the term
in square brackets can easily be  evaluated using the available CG related quantities. In particular, the lower bound is given by $\gamma_k \| r_k \|^2$, and the upper bound
by $\gamma_{k}^{(\mu)}\|r_{k}\|^{2}$, where
$\gamma_{k}^{(\mu)}$ can be updated using \eqref{eq:gammamu}.
}

To summarize results of \cite{GoMe1994,GoSt1994,StTi2002}, and
\cite{MeTi2013,MeTi2014,MeTi2019} related to the Gauss and Gauss-Radau
quadrature bounds for the $A$-norm of the error in~CG, it has been
shown that
\begin{equation}
\gamma_{k}\|r_{k}\|^{2}\le\|x-x_{k}\|_{A}^{2}<\gamma_{k}^{(\mu)}\|r_{k}\|^{2}<\left(\frac{\|r_{k}\|^{2}}{\mu\|p_{k}\|^{2}}\right)\|r_{k}\|^{2}\label{eq:basic}
\end{equation}
for $k<n-1$, and $\mu$ such that $0<\mu\leq\lambda_{1}$. Note that
in the special case $k=n-1$ it holds that $\|x-x_{n-1}\|_{A}^{2}=\gamma_{n-1}\|r_{n-1}\|^{2}$.
If the initial residual $r_{0}$ has a nontrivial component in the
eigenvector corresponding to $\lambda_{1}$, then $\lambda_{1}$ is
an eigenvalue of $T_{n}$. If in addition $\mu$ is chosen such that
$\mu=\lambda_{1}$, then $\gamma_{n-1}=\gamma_{n-1}^{(\mu)}$ and
the second inequality in \eqref{eq:basic} changes to equality.
The last inequality is strict also for $k=n-1$.

The rightmost bound in \eqref{eq:basic},
{that will be called the \emph{simple upper bound} in the following,}
was derived in \cite{MeTi2019}.
The norm $\Vert p_k\Vert$ is not available in CG, but the ratio
\[
\phi_{k}=\frac{\left\Vert r_{k}\right\Vert ^{2}}{\left\Vert p_{k}\right\Vert ^{2}}
\]
can be computed efficiently using
\begin{equation}
\phi_{j+1}^{-1}=1+\phi_{j}^{-1}\delta_{j+1},\qquad\phi_{0}=1.\label{eq:phiupdate}
\end{equation}
{Note that at an iteration $\ell\leq k$ we can obtain a more accurate bound}
using
\begin{equation}
\|x-x_{\ell}\|_{A}^{2}=\sum_{j=\ell}^{k-1}\gamma_{j}\left\Vert r_{j}\right\Vert ^{2}+\|x-x_{k}\|_{A}^{2},\label{eq:delay}
\end{equation}
by applying the basic bounds
\eqref{eq:basic} to the last term in \eqref{eq:delay}; see \cite{MeTi2019}
for details on the construction of more accurate bounds. In practice,
however, one runs the CG algorithm, and estimates the error in a backward
way, i.e., $k-\ell$ iterations back. The adaptive choice of the delay
$k-\ell$ when using the Gauss quadrature lower bound was discussed recently
in \cite{MePaTi2021}.

In the following we will we concentrate on the analysis of the behaviour
of the basic Gauss-Radau upper {bound
\begin{equation}
\gamma_{k}^{(\mu)}\|r_{k}\|^{2}\label{eq:GR}
\end{equation}
in dependence of the choice of $\mu$. As already mentioned,
we observed in many examples that in earlier iterations,
the bound is approximating the squared $A$-norm of the error quite well,
but in later iterations it becomes worse,
it is delayed and almost independent of $\mu$.
We observed that this phenomenon  is related to the convergence
of the smallest Ritz value to the smallest eigenvalue $\lambda_1$.
In particular, the bound is getting worse if
the smallest Ritz value approximates $\lambda_1$ better than $\mu$.
This often happens during finite precision computations when convergence of CG is delayed because of rounding errors and there are clusters of Ritz values approximating individual eigenvalues of $A$.
Usually, such clusters arise around the largest eigenvalues.
At some iteration, each eigenvalue of $A$ can be approximated by a Ritz value,
while the $A$-norm of the error still does not reach the required level of accuracy,
and the process will continue and place more and more Ritz values in the clusters. In this situation, it can happen that $\lambda_1$
is tightly (that is, to a high relative accuracy) approximated by a Ritz value  while the CG process still continues. Note that if $A$ has well separated eigenvalues and we run the experiment in exact arithmetic, then $\lambda_1$ is usually
tightly approximated by a Ritz value only in the last iterations.
The above observation is key for constructing a motivating example,
in which we can readily observe the studied phenomenon also
in exact arithmetic, and which will motivate our analysis.}

\section{The model problem and a numerical experiment}

\label{sec:model}

In the construction of the motivating example we use results presented
in \cite{Gr1989,GrSt1992,B:Me2006,OLStTi2007,St1991}. Based on the
work by Chris Paige \cite{Pa1980a}, Anne Greenbaum \cite{Gr1989}
proved that the results of finite precision CG computations can be
interpreted (up to some small inaccuracies) as the results of the
exact CG algorithm applied to a larger system with the system matrix
having many eigenvalues distributed throughout ``tiny'' intervals
around the eigenvalues of the original matrix. The experiments show
that ``tiny'' means of the size comparable to $\mathbf{u}\|A\|$,
where $\mathbf{u}$ is the roundoff unit. This result was used in
\cite{GrSt1992} to predict the behavior of finite precision CG. Inspired
by \cite{Gr1989,GrSt1992,OLStTi2007} we will construct a linear system
$Ax=b$ with similar properties as the one suggested by Greenbaum~\cite{Gr1989}.
However, we want to emphasize and visualize some phenomenons concerning
the behaviour of the basic Gauss-Radau upper bound \eqref{eq:GR}.
Therefore, we choose the size of the intervals around the original
eigenvalues larger than $\mathbf{u}\|A\|$.

We start with the test problem $\Lambda y=w$ from \cite{St1991},
where $w=m^{-1/2}(1,\dots,1)^{T}$ and $\Lambda=\mathrm{diag}(\hat{\lambda}_{1},\dots,\hat{\lambda}_{m})$,
\begin{equation}
\hat{\lambda}_{i}=\hat{\lambda}_{1}+\frac{i-1}{m-1}(\hat{\lambda}_{m}-\hat{\lambda}_{1})\rho^{m-i},\quad i=2,\ldots,m.\label{eq:strakos}
\end{equation}
The diagonal matrix $\Lambda$ and the vector $w$ determine the stepwise
distribution function $\omega(\lambda)$ with points of increase $\hat{\lambda}_{i}$
and the individual jumps (weights) $\omega_{j}=m^{-1}$,

\begin{equation}
\omega(\lambda)\equiv\;\left\{ \;\begin{array}{ccl}
0 & \textnormal{for} & \lambda<\hat{\lambda}_{1}\,,\\[1mm]
\sum_{j=1}^{i}\omega_{j} & \textnormal{for} & \hat{\lambda}_{i}\leq\lambda<\hat{\lambda}_{i+1}\,,\quad1\leq i\leq m-1\,,\\[1mm]
1 & \textnormal{for} & \hat{\lambda}_{m}\leq\lambda\,.
\end{array}\right.\,\label{eq:schema_hermit}
\end{equation}

We construct a blurred distribution function $\widetilde{\omega}(\lambda)$
having clusters of points of increase around the original eigenvalues
$\hat{\lambda}_{i}$. We consider each cluster to have the same radius
$\delta$, and let the number $c_i$ of points in the $i$th cluster
grow linearly from 1 to $p$,
\[
c_{i}=\mathrm{round}\left(\frac{p-1}{m-1}i+\frac{m-p}{m-1}\right),\quad i=1,\dots,m.
\]
The blurred eigenvalues
\[
\widetilde{\lambda}_{j}^{(i)},\quad j=1,\dots,c_{i},
\]
are uniformly distributed in $[\hat{\lambda}_{i}-\delta,\hat{\lambda}_{i}+\delta]$,
with the corresponding weights given by
\[
\widetilde{\omega}_{j}^{(i)}=\frac{\omega_{i}}{c_{i}}\quad j=1,\dots,c_{i},
\]
i.e., the weights that correspond to the $i$th cluster are equal,
and their sum is $\omega_{i}$. Having defined the blurred distribution
function $\widetilde{\omega}(\lambda)$ we can construct the corresponding
Jacobi matrix $T\in{\mathbb R}^{N\times N}$
{in a numerically stable way}
using the Gragg and Harrod
rkpw algorithm \cite{GrHa1984}.
{Note that the mapping
from the nodes and weights of the computed quadrature to the recurrence
coefficients
is generally well-conditioned \cite[p.~59]{B:Gaut2004}}.
To construct the above mentioned Jacobi matrix $T$ we used Matlab's
vpa arithmetic with 128 digits. Finally, we define the \emph{double
precision data} $A$ and $b$ that will be used for experimenting
as
\begin{equation}
A=\mathrm{double}(T),\quad b=e_{1},\label{eq:Abdouble}
\end{equation}
where $e_{1}\in{\mathbb R}^{N}$ is the first column of the identity matrix.
We decided to use double precision input data
{since
we can easily compare results of our computations
performed in Matlab's vpa arithmetic with results obtained using double
precision arithmetic for the same input data.}

In our experiment we choose $m=12$, $\hat{\lambda}_{1}=10^{-6}$,
$\hat{\lambda}_{m}=1$, $\rho=0.8$, $\delta=10^{-10}$, and $p=4$,
resulting in $N=30.$ Let us run the ``exact'' CGQ
algorithm of \cite{MeTi2013} on the model problem \eqref{eq:Abdouble}
constructed above, where exact arithmetic is simulated using Matlab's
variable precision with \texttt{digits=128}. Let
$\lambda_{1}$ be the exact smallest eigenvalue of $A$. We use four
different values of $\mu$ for the computation of the Gauss-Radau
upper bound \eqref{eq:GR}: $\mu_{3}=(1-10^{-3})\lambda_{1}$, $\mu_{8}=(1-10^{-8})\lambda_{1}$,
$\mu_{16}$ which denotes the double precision number closest to $\lambda_{1}$
such that $\mu_{16}\leq\lambda_{1}$, and $\mu_{50}=(1-10^{-50})\lambda_{1}$
which is almost like the exact value. Note that $\gamma_{k}^{(\mu)}$
is updated using \eqref{eq:gammamu}.


Figure~\ref{fig-0} shows the $A$-norm of the error $\|x-x_{k-1}\|_{A}$
(solid curve), the upper bounds for the considered values of $\mu_{i}$,
$i=3,8,16,50$ (dotted solid curves), {and
the rightmost bound in \eqref{eq:basic} (the simple upper bound)
for $\mu_{50}$ (dashed curve).} The dots represent the values
$\theta_{1}^{(k)}-\lambda_{1}$, i.e., the distances of the smallest
Ritz values $\theta_{1}^{(k)}$ to $\lambda_{1}$. The horizontal
dotted lines correspond to the values of $\lambda_{1}-\mu_{i}$, $i=3,8,16$.

\begin{figure}[!htbp]
\centering{}\includegraphics[width=11cm]{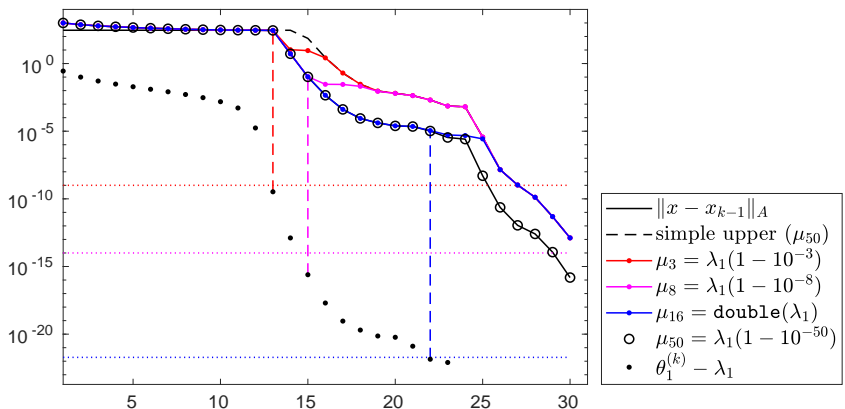} \caption{$\|x-x_{k-1}\|_{A}$, upper bounds and the distance of $\theta_{1}^{(k)}$
to $\lambda_{1}$, {\tt digits=128}.}
\label{fig-0}
\end{figure}


The {Gauss-Radau} upper bounds in Figure~\ref{fig-0} first overestimate, and then
closely approximate $\|x-x_{k-1}\|_{A}$ (starting from iteration
5). However, at some point, the {Gauss-Radau}
upper bounds start to differ significantly
from $\|x-x_{k-1}\|_{A}$ and represent {worse approximations},
except for $\mu_{50}$.
We observe that for a given $\mu_{i}$, $i=3,8,16$,
the upper bounds are delayed when the distance of $\theta_{1}^{(k)}$
to $\lambda_{1}$ becomes smaller than the distance of $\mu_{i}$
to $\lambda_{1}$, i.e., when
\begin{equation}
\theta_{1}^{(k)}-\lambda_{1}<\lambda_{1}-\mu_{i}.\label{eq:first}
\end{equation}
If \eqref{eq:first} holds, then the smallest Ritz value $\theta_{1}^{(k)}$
is a better approximation to $\lambda_{1}$ than~$\mu_{i}$. This
moment is emphasized using vertical dashed lines that connect the
value $\theta_{1}^{(k)}-\lambda_{1}$ with $\|x-x_{k-1}\|_{A}$ in
the first iteration $k$ such that \eqref{eq:first} holds. Moreover,
below a certain level, the upper bounds become almost independent
of $\mu_{i}$, $i=3,8,16$, and
{visually coincide with
the simple upper bound.}
The closer is $\mu$ to $\lambda_1$, the later
this phenomenon occurs.

Depending on the validity of \eqref{eq:first}, we distinguish between
\emph{phase~1} and\emph{ phase~2 }of convergence. If the inequality
\eqref{eq:first} does not hold, i.e., if $\mu$ is a better approximation
to $\lambda_{1}$ than the smallest Ritz value, then we say we are
in phase~1.
If \eqref{eq:first} holds,
then the smallest Ritz value is closer to $\lambda_{1}$ than $\mu$
and we are in phase~2.
Obviously,
the beginning of phase 2 depends on the choice of $\mu$ and on the convergence of the smallest Ritz value to the smallest eigenvalue. Note that for $\mu=\mu_{50}$ we are always in phase~1 before we stop the iterations.

{In the given experiment as well as in many practical problems, the delay of the upper bounds is not large (just a few iterations), and the bounds can still provide a useful information for stopping the algorithm. However, we have also encountered examples where
the delay of the Gauss-Radau upper bound  was about 200 iterations; see, e.g.,
\cite[Fig.~10]{MeTi2019}
or \cite[Fig.~2]{MePaTi2021} concerning the matrix {\tt s3dkt3m2}.
Hence, we believe that this phenomenon deserves attention and explanation.}

%
%
%
%
%
%
%
%
%
%
%
%
%
%
%

\section{Analysis}
\label{sec:analysis}
The upper bounds are computed from the modified tridiagonal matrices
\eqref{eq:Tmu} discussed in Section~\ref{sec:modified}, that differ
only in one coefficient at the position $(k+1,k+1)$. Therefore, the
first step of the analysis is to understand how the choice of $\mu$
and the validity of the condition \eqref{eq:first} influences the
value of the modified coefficient
\begin{eqnarray}
\alpha_{k+1}^{(\mu)} & = & \mu+\sum_{i=1}^{k}\eta_{i,k}^{(\mu)},\qquad\eta_{i,k}^{(\mu)}=\frac{\left(\beta_{k}s_{k,i}^{(k)}\right)^{2}}{\theta_{i}^{(k)}-\mu};\label{eq:terms}
\end{eqnarray}
see \eqref{eq:alphaspectral}. We will compare its value to a modified
coefficient for which phase~2 does not occur; see Figure~\ref{fig-0}
for $\mu_{50}$.

Based on that understanding we will then address further important
questions. First, our aim is to explain the behaviour of the basic
Gauss-Radau upper bound \eqref{eq:GR} in phase~2, in particular,
{its closeness to the simple upper bound \eqref{eq:basic}}. Second,
for practical reasons, without knowing $\lambda_{1}$, we would like to be
able to detect phase~2, i.e., the first iteration $k$ for which
the inequality \eqref{eq:first} starts to hold. Finally, we address
the problem of how to improve the accuracy of the basic Gauss-Radau
upper bound \eqref{eq:GR} in phase 2.


We first analyze the relation between two modified coefficients $\alpha_{k+1}^{(\mu)}$
and $\alpha_{k+1}^{(\lambda)}$ where $0<\mu<\lambda<\theta_{1}^{(k)}.$


\begin{lemma}
\label{lem:lemma2}Let $0<\mu<\lambda<\theta_{1}^{(k)}$. Then
\begin{equation}
\frac{\eta_{i,k}^{(\lambda)}-\eta_{i,k}^{(\mu)}}{\eta_{i,k}^{(\mu)}}=\frac{\lambda-\mu}{\theta_{i}^{(k)}-\lambda}\label{eq:sensitivity}
\end{equation}
and
\begin{equation}
\alpha_{k+1}^{(\lambda)}-\alpha_{k+1}^{(\mu)}=\left(\frac{\lambda-\mu}{\theta_{1}^{(k)}-\mu}\right)\eta_{1,k}^{(\lambda)}+\left(\lambda-\mu\right)E_{k}^{(\lambda,\mu)}\label{eq:diff1}
\end{equation}
where
\begin{equation}
E_{k}^{(\lambda,\mu)}\equiv1+\sum_{i=2}^{k}\frac{\eta_{i,k}^{(\lambda)}}{\theta_{i}^{(k)}-\mu}\label{eq:Ek}
\end{equation}
satisfies $E_{k}^{(\lambda,\mu)}=E_{k}^{(\mu,\lambda)}.$
\end{lemma}

\begin{proof} {
From the definition of $\eta_{i,k}^{(\mu)}$ and $\eta_{i,k}^{(\lambda)}$,
it follows immediately
\[
\frac{\eta_{i,k}^{(\lambda)}}{\theta_{i}^{(k)}-\mu}=\frac{\eta_{i,k}^{(\mu)}}{\theta_{i}^{(k)}-\lambda},
\]
which implies $E_{k}^{(\lambda,\mu)}=E_{k}^{(\mu,\lambda)}$ and
\eqref{eq:sensitivity}.

Note that $0<\eta_{i,k}^{(\mu)}<\eta_{i,k}^{(\lambda)}$. Using \eqref{eq:sensitivity}, the difference of the coefficients $\alpha$'s is
\begin{eqnarray*}
\alpha_{k+1}^{(\lambda)}-\alpha_{k+1}^{(\mu)} & = & \left(\lambda-\mu\right)+\sum_{i=1}^{k}\left(\eta_{i,k}^{(\lambda)}-\eta_{i,k}^{(\mu)}\right)\\
 & = & {
(\lambda-\mu)+(\lambda-\mu)\sum_{i=1}^{k}\frac{\eta_{i,k}^{(\mu)}}{\theta_{i}^{(k)}-\lambda}} \\
%
 & = &
{ (\lambda-\mu)\frac{\eta_{1,k}^{(\lambda)}}{\theta_{1}^{(k)}-\mu} +
(\lambda-\mu)\left(1+\sum_{i=2}^{k}\frac{\eta_{i,k}^{(\mu)}}{\theta_{i}^{(k)}-\lambda}\right) }
\end{eqnarray*}
which implies \eqref{eq:diff1}. }
\end{proof}

\subsection{Assumptions}\label{sec:assumptions}
{Let us describe in more detail the situation we are interested in. In the analysis that follows we will assume implicitly the following.
\begin{enumerate}
\item  $\lambda_{1}$ is well separated from $\lambda_{2}$ so that we can use the gap theorem mentioned in Section~\ref{sec:eigenvalues}, in particular relation
\eqref{eq:bound1} bounding $\eta_{1,k}^{(\lambda_1)}$.

\item $\mu$ is a tight underestimate to $\lambda_{1}$ such that
\begin{equation}
\lambda_{1}-\mu\ll\lambda_{2}-\lambda_{1}.\label{eq:separate}
\end{equation}

\item The smallest Ritz value $\theta_1^{(k)}$ converges to $\lambda_1$ with increasing $k$ so that there
 is an iteration index $k$ from which
$$
\theta_1^{(k)}-\lambda_1 \ll \lambda_1-\mu.
$$
\end{enumerate}

Let us briefly comment on these assumptions. The assumption
that $\lambda_{1}$ is well separated from $\lambda_{2}$ is used later
to prove that $\eta_{1,k}^{(\lambda_1)}$ is bounded away from zero;
see \eqref{eq:boundedaway}.
If there is a cluster of eigenvalues around $\lambda_1$, one can still observe the discussed phenomenon of loss of accuracy of the upper bound, but a theoretical analysis would be much more complicated. Note that the first assumption is also often satisfied
for a system matrix $\hat{A}$ that models finite precision CG behaviour,
if the original matrix $A$ has well separated eigenvalues $\lambda_1$ and $\lambda_2$.
Using results of Greenbaum \cite{Gr1989} we know that $\hat{A}$ can have
many eigenvalues distributed throughout tiny intervals around the eigenvalues of $A$. We have constructed the model matrix $\hat{A}$ in many numerical experiments,
using the procedure suggested in \cite{Gr1989}. We found out that the constructed $\hat{A}$ has usually clusters of eigenvalues around the larger eigenvalues of $A$ while
a smaller eigenvalue of $A$ is usually approximated by just one eigenvalue of $\hat{A}$. Therefore, the analysis presented below can then be applied to the matrix $\hat{A}$ that models the finite precision CG behavior.


If $\mu$ is not a tight underestimate, then the Gauss-Radau upper bound
is usually not a very good approximation of the $A$-norm of the error. Then the condition \eqref{eq:first} can hold from the beginning and
phase~1 need not happen.

Finally, in theory, the smallest Ritz value need not converge to $\lambda_1$
until the last iteration \cite{Sc1979}. But, in that case, there won't be any problem for the Gauss-Radau upper bound. However, in practical computations,
we very often observe the convergence of $\theta_1^{(k)}$ to $\lambda_1$. In particular,
in cases of matrices $\hat{A}$ with clustered eigenvalues that model finite
precision behavior of CG,
$\theta_1^{(k)}$ approximates $\lambda_1$ to a high relative accuracy usually earlier before the $A$-norm of the error reaches the ultimate level of accuracy.}

\subsection{{The modified coefficient $\alpha_{k+1}^{(\mu)}$}}
\label{sec:alphamu}

{Below we} would like to compare $\alpha_{k+1}^{(\lambda_{1})}$ for which
phase~2 does not occur with $\alpha_{k+1}^{(\mu)}$for which phase~2
occurs; see Figure~\ref{fig-0}. Using \eqref{eq:sensitivity} and
\eqref{eq:separate}, we are able to compare the individual $\eta$-terms.
In particular, for $i>1$ we {get}
\[
\frac{\eta_{i,k}^{(\lambda_{1})}-\eta_{i,k}^{(\mu)}}{\eta_{i,k}^{(\mu)}}=
{\frac{\lambda_{1}-\mu}{\theta_{i}^{(k)}-\lambda_{1}}}
<\frac{\lambda_{1}-\mu}{\lambda_{2}-\lambda_{1}}\ll1,
\]
{where we have used $\theta_i^{(k)} > \lambda_2$ for $i>1$. Therefore,}
\[
\eta_{i,k}^{(\lambda_{1})}\approx\eta_{i,k}^{(\mu)}\quad\mbox{for}\quad i>1.
\]
Hence, $\alpha_{k+1}^{(\mu)}$ can significantly differ from $\alpha_{k+1}^{(\lambda_{1})}$
only in the first term of the sum in \eqref{eq:terms} for which
\begin{equation}
\frac{\eta_{1,k}^{(\lambda_{1})}-\eta_{1,k}^{(\mu)}}{\eta_{1,k}^{(\mu)}}=\frac{\lambda_{1}-\mu}{\theta_{1}^{(k)}-\lambda_{1}}.\label{eq:compare1}
\end{equation}

If $\theta_{1}^{(k)}$ is a better approximation to $\lambda_{1}$
than $\mu$ in the sense of \eqref{eq:first},
then \eqref{eq:compare1} shows that $\eta_{1,k}^{(\lambda_{1})}$
can be much larger than $\eta_{1,k}^{(\mu)}$. As a consequence, $\alpha_{k+1}^{(\lambda_{1})}$
can differ significantly from $\alpha_{k+1}^{(\mu)}$ . On the other
hand, if $\mu$ is chosen such that
\[
\lambda_{1}-\mu\ll\theta_{1}^{(k)}-\lambda_{1},
\]
for all $k$ we are interested in, then phase 2 will not occur, and
{
\[
\alpha_{k+1}^{(\lambda_1)}-\alpha_{k+1}^{(\mu)}  =  \left(\lambda_1-\mu\right)+\sum_{i=1}^{k}\left(\eta_{i,k}^{(\lambda_1)}-\eta_{i,k}^{(\mu)}\right)
\approx
0\,,
\]
since $\mu$ is assumed to be a tight approximation to $\lambda_1$ and
$\eta_{i,k}^{(\lambda_1)}\approx \eta_{i,k}^{(\mu)}$ for all~$i$.}

In the following we discuss phase~1 and phase~2 in more detail.
%
%
%
In phase~1,
\[
\lambda_{1}-\mu<\theta_{1}^{(k)}-\lambda_{1},
\]
and, therefore, all components $\eta_{i,k}^{(\mu)}$ (including $\eta_{1,k}^{(\mu)}$)
are not sensitive to small changes of $\mu$; see \eqref{eq:sensitivity}.
In other words, the coefficients $\alpha_{k+1}^{(\mu)}$ are approximately
the same for various choices of $\mu$.

Let us denote
\[
{h_k}=\frac{\lambda_{1}-\mu}{\theta_{1}^{(k)}-\lambda_{1}}<1.
\]
In fact, we can write $\theta_{1}^{(k)}-\mu=\theta_{1}^{(k)}-\lambda_{1}+\lambda_{1}-\mu$
and use the Taylor expansion of $1/(1+{h_k})$. It yields
\begin{eqnarray*}
\frac{1}{\theta_{1}^{(k)}-\mu} & = & \frac{1}{\theta_{1}^{(k)}-\lambda_{1}}{\left(\frac{1}{h_k+1}\right)=\frac{1}{\theta_{1}^{(k)}-\lambda_{1}}\left[1-h_k+h_k^{2}-h_k^{3}+\cdots\right]}.
\end{eqnarray*}
{Obviously, $h_k$ is an increasing function of the iteration number $k$; the numerator
is constant while the denominator is decreasing in absolute value. The size of $h_k$ depends
also on how well $\mu$ approximates $\lambda_1$. If $\mu$ is a tight approximation
to $\lambda_1$, then, at the beginning of the CG iterations, the denominator of $h_k$
can be large compared to the numerator, $h_k$} is small and the
right-hand side of $1/(\theta_{1}^{(k)}-\mu)$ is almost given by
$1/(\theta_{1}^{(k)}-\lambda_{1})$, independent of $\mu$. {We observed
that the first term of the sum of the $\eta_{i,k}^{(\mu)}$ is then usually the largest
one.}


{Let us now discuss phase~2.}
First recall that for any $0<\mu<\lambda_{1}$ it holds that
\begin{equation}
\alpha_{k+1}^{(\mu)}<\alpha_{k+1}^{(\lambda_{1})}\quad\mbox{and}\quad\eta_{1,k}^{(\mu)}<\eta_{1,k}^{(\lambda_{1})}.\label{eq:coefficients}
\end{equation}
As before, suppose that $\lambda_{1}$ is well separated from $\lambda_{2}$
and that \eqref{eq:separate} holds. Phase~2 begins when $\theta_{1}^{(k)}$
is a better approximation to $\lambda_{1}$ than $\mu$, i.e., when
\eqref{eq:first} holds. Since $\theta_{1}^{(k)}$ is a tight approximation
to $\lambda_{1}$ in phase~2, \eqref{eq:bound1} and
\eqref{eq:first} imply that
\begin{equation}\label{eq:boundedaway}
\eta_{1,k}^{(\lambda_{1})}\geq\lambda_{2}-\theta_{1}^{(k)}
= \lambda_{2}-\lambda_{1} + \lambda_1-\theta_{1}^{(k)} >
 (\lambda_{2}-\lambda_{1}) - (\lambda_1-\mu).
\end{equation}
Therefore, using \eqref{eq:separate},
$\eta_{1,k}^{(\lambda_{1})}$ is bounded away from zero.
On the other hand, \eqref{eq:bound1} also implies that
\[
\eta_{1,k}^{(\mu)}=\frac{\theta_{1}^{(k)}-\lambda_{1}}{\theta_{1}^{(k)}-\mu}\eta_{1,k}^{(\lambda_{1})}\leq\frac{\theta_{1}^{(k)}-\lambda_{1}}{\theta_{1}^{(k)}-\mu}\left(\lambda_{n}-\lambda_{1}\right)
\]
and as $\theta_{1}^{(k)}$ converges to $\lambda_{1}$, $\eta_{1,k}^{(\mu)}$
goes to zero. Therefore,
\[
\alpha_{k+1}^{(\mu)}\approx\mu+\sum_{i=2}^{k}\eta_{i,k}^{(\mu)},
\]
and the sum on the right-hand side is almost independent of $\mu$. Note that having two values $0<\mu<\lambda<\lambda_{1}$ such that
\begin{equation}
\theta_{1}^{(k)}-\lambda_{1}<\lambda_{1}-\lambda\quad\mbox{and}\quad\lambda-\mu\ll\lambda_{2}-\lambda_{1},\label{eq:assumption2}
\end{equation}
then one can expect that
\begin{equation}
\alpha_{k+1}^{(\mu)}\approx\alpha_{k+1}^{(\lambda)}\label{eq:predict-1}
\end{equation}
because $\eta_{1,k}^{(\mu)}$ and $\eta_{1,k}^{(\lambda)}$ converge
to zero and $\eta_{i,k}^{(\mu)}\approx\eta_{i,k}^{(\lambda)}$ for
$i>1$ due to
\[
\frac{\eta_{i,k}^{(\lambda)}-\eta_{i,k}^{(\mu)}}{\eta_{i,k}^{(\mu)}}=\frac{\lambda-\mu}{\theta_{i}^{(k)}-\lambda}<\frac{\lambda-\mu}{\lambda_{2}-\lambda_{1}}\ll1,
\]
where we have used \eqref{eq:sensitivity} and the assumption \eqref{eq:assumption2}.
Therefore, $\alpha_{k+1}^{(\mu)}$ is relatively insensitive to small
changes of $\mu$ and the same is true for the upper bound \eqref{eq:GR}.

\subsection{The coefficient $\alpha_{k+1}$\label{subsec:alpha}}

The coefficient $\alpha_{k+1}$ can also be written as
\[
\alpha_{k+1}=\alpha_{k+1}^{(\mu)}\quad\mbox{for}\quad\mu=\theta_{1}^{(k+1)},
\]
and the results of Lemma~\ref{lem:alphamu} and Lemma~\ref{lem:lemma2}
are still valid, even though, in practice, $\mu$ must be smaller
than $\lambda_{1}$. Using \eqref{eq:diff1} we can express the differences
between the coefficients, it holds that
\begin{eqnarray}
\alpha_{k+1}-\alpha_{k+1}^{(\lambda_{1})} & = & \eta_{1,k}^{(\lambda_{1})}\frac{\theta_{1}^{(k+1)}-\lambda_{1}}{\theta_{1}^{(k)}-\theta_{1}^{(k+1)}}+\left(\theta_{1}^{(k+1)}-\lambda_{1}\right)E_{k}^{(\theta_{1}^{(k+1)},\lambda_{1})}.\label{eq:diff02}
\end{eqnarray}
If the smallest Ritz value $\theta_{1}^{(k+1)}$ is close to $\lambda_{1}$,
then the second term of the right-hand side in \eqref{eq:diff02}
will be negligible in comparison to the first one, since
\[
E_{k}^{(\theta_{1}^{(k+1)},\lambda_{1})}=\mathcal{O}(1),
\]
see \eqref{eq:Ek}, and since $\eta_{1,k}^{(\lambda_{1})}$ is bounded
away from zero; see \eqref{eq:boundedaway}. Therefore, one can expect
that
\begin{equation}
\alpha_{k+1}-\alpha_{k+1}^{(\lambda_{1})}\ \approx\ \eta_{1,k}^{(\lambda_{1})}\frac{\theta_{1}^{(k+1)}-\lambda_{1}}{\theta_{1}^{(k)}-\theta_{1}^{(k+1)}}.\label{eq:closeness}
\end{equation}
The size of the term on the right-hand side is related to the speed
of convergence of the smallest Ritz value $\theta_{1}^{(k)}$ to $\lambda_{1}$.
Denoting
\[
\frac{\theta_{1}^{(k+1)}-\lambda_{1}}{\theta_{1}^{(k)}-\lambda_{1}}=\rho_{k}<1,
\]
we obtain
\[
\frac{\theta_{1}^{(k+1)}-\lambda_{1}}{\theta_{1}^{(k)}-\theta_{1}^{(k+1)}}=\frac{\frac{\theta_{1}^{(k+1)}-\lambda_{1}}{\theta_{1}^{(k)}-\lambda_{1}}}{1-\frac{\theta_{1}^{(k+1)}-\lambda_{1}}{\theta_{1}^{(k)}-\lambda_{1}}}=\frac{\rho_{k}}{1-\rho_{k}}.
\]
For example, if the convergence of $\theta_{1}^{(k)}$ to $\lambda_{1}$
is superlinear, i.e., if $\rho_{k}\rightarrow0$, {then $\alpha_{k+1}$
and $\alpha_{k+1}^{(\lambda_{1})}$ are close.}

\subsection{{Numerical experiments}}

Let us demonstrate {numerically} the theoretical results {described in previous sections} using our model problem. To compute
the following results, we, again, use Matlab's vpa arithmetic with 128 {decimal} 
digits.

\begin{figure}[!htbp]
\centering{}\includegraphics[width=11cm]{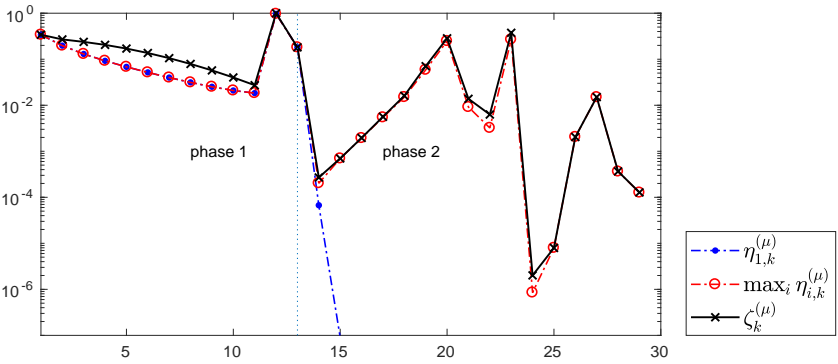} \caption{First term $\eta_{1,k}^{(\mu)}$, maximum term $\eta_{i,k}^{(\mu)}$,
and the sum $\zeta_{k}^{(\mu)}$ for $\mu=\mu_{3}$.}
\label{fig-7}
\end{figure}

We first consider $\mu=\mu_{3}=(1-10^{-3})\lambda_{1}$ for which
we have $\lambda_{1}-\mu=10^{-9}$. The switch from phase 1 to phase
2 occurs at iteration 13. Figure~\ref{fig-7} displays the first
term $\eta_{1,k}^{(\mu)}$ and the maximum term $\eta_{i,k}^{(\mu)}$
as well as the sum $\zeta_{k}^{(\mu)}$ defined by~\eqref{eq:zeta},
see Lemma~\ref{lem:alphamu}, as a function of the iteration number
$k$. In phase 1 the first term $\eta_{1,k}^{(\mu)}$ is the largest one. As predicted, after the start of phase 2, the first term is
decreasing quite fast.

\begin{figure}[!htbp]
\centering{}\includegraphics[width=11cm]{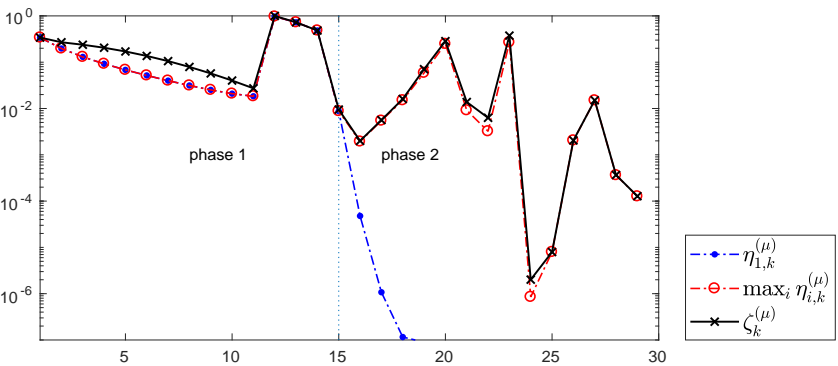} \caption{First term $\eta_{1,k}^{(\mu)}$, maximum term $\eta_{i,k}^{(\mu)}$,
and the sum $\zeta_{k}^{(\mu)}$ for $\mu=\mu_{8}$.}
\label{fig-4}
\end{figure}

Let us now use $\mu=\mu_{8}=(1-10^{-8})\lambda_{1}$ for which we
have $\lambda_{1}-\mu=10^{-14}$. The switch from phase 1 to phase
2 occurs at iteration 15; see Figure~\ref{fig-4}. The conclusions
are the same as for $\mu_{3}$.

\begin{figure}[!htbp]
\centering{}\includegraphics[width=11cm]{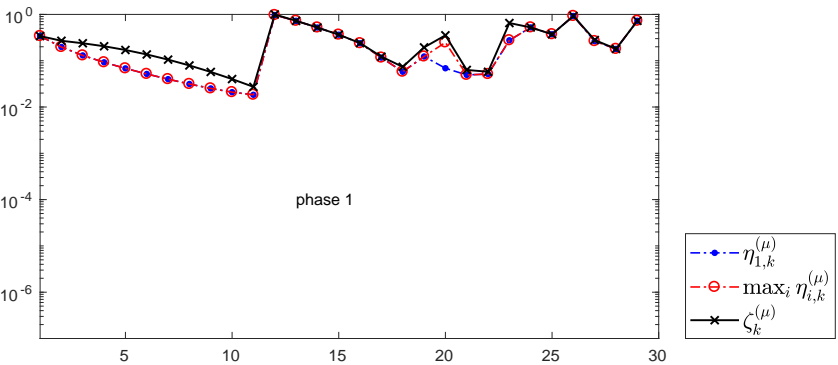} \caption{First term $\eta_{1,k}^{(\mu)}$, maximum term $\eta_{i,k}^{(\mu)}$,
and the sum $\zeta_{k}^{(\mu)}$ for $\mu=\mu_{50}$.}
\label{fig-3}
\end{figure}

The behavior of the first term is completely different for $\mu=(1-10^{-50})\lambda_{1}$
which almost corresponds to using the exact smallest eigenvalue $\lambda_{1}$.
The maximum term of the sum is then
almost always the first one; see Figure~\ref{fig-3}. Remember that,
for this value of $\mu$, we are always in phase 1.

\begin{figure}[!htbp]
\centering{}\includegraphics[width=11cm]{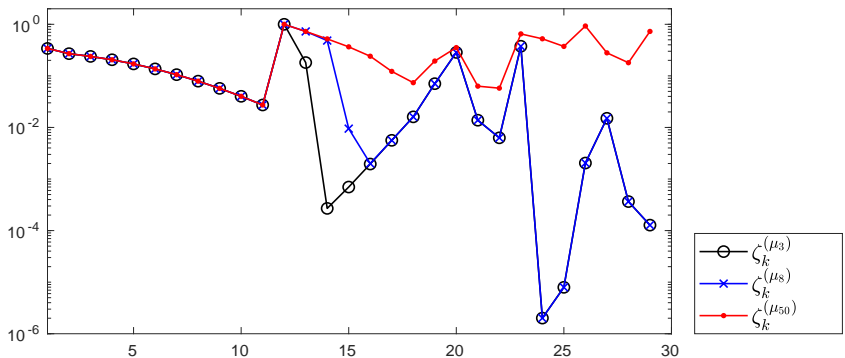} \caption{Comparison of the sums $\zeta_{k}^{(\mu_{3})}$, $\zeta_{k}^{(\mu_{8})}$,
and $\zeta_{k}^{(\mu_{50})}$.}
\label{fig-3-1}
\end{figure}

Finally, in Figure~\ref{fig-3-1} we present a comparison of the sums
$\zeta_{k}^{(\mu)}$ for $\mu_{3}$, $\mu_{8}$, and $\mu_{50}$.
We observe that from the beginning up to iteration $12$, all sums
visually coincide. Starting from iteration 13 we enter phase~2 for
$\mu=\mu_{3}$ and the sum $\zeta_{k}^{(\mu_{3})}$ starts to differ
significantly from the other sums, in particular from the ``reference''
term $\zeta_{k}^{(\mu_{50})}$. Similarly, for $k=15$ we enter phase~2
for $\mu=\mu_{8}$ and $\zeta_{k}^{(\mu_{8})}$ and $\zeta_{k}^{(\mu_{50})}$
start to differ. We can also observe that $\zeta_{k}^{(\mu_{3})}$
and $\zeta_{k}^{(\mu_{8})}$ significantly differ only in iterations
13, 14, and 15, i.e., when we are in phase~2 for $\mu=\mu_{3}$ but
in phase~1 for $\mu=\mu_{8}$. In all other iterations, $\zeta_{k}^{(\mu_{3})}$
and $\zeta_{k}^{(\mu_{8})}$ visually coincide.

\begin{figure}[!htbp]
\centering{}\includegraphics[width=11cm]{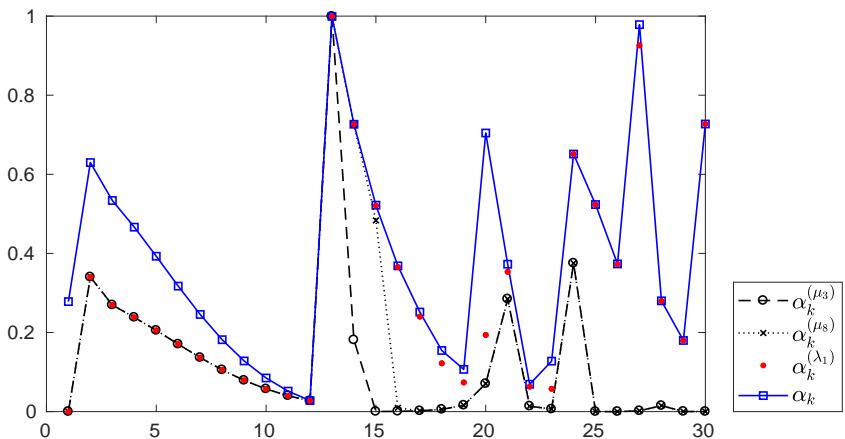} \caption{$\alpha_{k}^{(\mu_{3})}$, $\alpha_{k}^{(\mu_{8})}$, $\alpha_{k}^{(\lambda_{1})}$,
and $\alpha_{k}$.}
\label{fig:1}
\end{figure}

In Figure~\ref{fig:1} we plot the coefficients $\alpha_{k}^{(\mu_{3})}$, $\alpha_{k}^{(\mu_{8})}$, $\alpha_{k}^{(\lambda_{1})}$ and $\alpha_{k}$, so that we can compare the observed behaviour with the predicted one.
Phase~2 starts for $\mu_{3}$ at iteration 13, and for $\mu_{8}$
at iteration 15; see also Figure~\ref{fig-0}. For $k\leq13$
we observe that
\[
\alpha_{k}^{(\mu_{3})}\approx\alpha_{k}^{(\mu_{8})}\approx\alpha_{k}^{(\lambda_{1})}
\]
as explained in {Section~\ref{sec:alphamu}} and $\alpha_k$ is larger. For $k\geq16$, the
first terms $\eta_{1,k-1}^{(\mu_{3})}$ and $\eta_{1,k-1}^{(\mu_{8})}$
are close to zero, and, as explained in {Section~\ref{sec:alphamu}},
\[
\alpha_{k}^{(\mu_{3})}\approx\alpha_{k}^{(\mu_{8})}.
\]
For $k=14$ and $k=15$, $\alpha_{k}^{(\mu_{3})}$ and $\alpha_{k}^{(\mu_{8})}$
can differ significantly because $\alpha_{k}^{(\mu_{3})}$ is already
in phase~2 while $\alpha_{k}^{(\mu_{8})}$ is still in phase~1.

We can also observe that $\alpha_{k}$ can be very close to $\alpha_{k}^{(\lambda_{1})}$
when the smallest Ritz value $\theta_{1}^{(k)}$ is a tight approximation
to $\lambda_{1}$, i.e., in later iterations. We know that the closeness
of $\alpha_{k}$ to $\alpha_{k}^{(\lambda_{1})}$ depends on the speed
of convergence of the smallest Ritz value to $\lambda_{1}$; see \eqref{eq:closeness}
and the corresponding discussion.

\section{The Gauss-Radau bound in phase~2}

Our aim in this section is to investigate the relation between the
basic Gauss-Radau upper bound \eqref{eq:GR} and the simple upper
bound; see \eqref{eq:basic}. Recall the notation
\[
\phi_{k}=\frac{\left\Vert r_{k}\right\Vert ^{2}}{\left\Vert p_{k}\right\Vert ^{2}};
\]
see \eqref{eq:phiupdate}. In particular, we would like to explain
why the two bounds almost coincide in phase~2. Note that using \eqref{eq:gammalpha}
we obtain
\begin{equation}
\alpha_{k+1}^{(\mu)}=\left(\gamma_{k}^{(\mu)}\right)^{-1}+\frac{\delta_{k}}{\gamma_{k-1}}\label{eq:alphamu}
\end{equation}
and from \eqref{eq:alphaspectral} it follows
\begin{eqnarray*}
\alpha_{k+1}^{(\mu)} & = & \mu+\beta_{k}^{2}e_{k}^{T}\left(T_{k}-\mu I\right)^{-1}e_{k},\qquad\beta_{k}^{2}=\frac{1}{\gamma_{k-1}}\frac{\delta_{k}}{\gamma_{k-1}}.
\end{eqnarray*}
Therefore,
\begin{eqnarray}
\left(\gamma_{k}^{(\mu)}\right)^{-1} & = & \mu+\beta_{k}^{2}\left(e_{k}^{T}\left(T_{k}-\mu I\right)^{-1}e_{k}-\gamma_{k-1}\right).\label{eq:gammatilde}
\end{eqnarray}
In the following lemma we {give} another expression for $e_{k}^{T}\left(T_{k}-\mu I\right)^{-1}e_{k}$.

\begin{lemma}
Let $0<\mu<\theta_{1}^{(k)}$. Then it holds that
\begin{equation}
e_{k}^{T}\left(T_{k}-\mu I\right)^{-1}e_{k}=\gamma_{k-1}+\mu\frac{\gamma_{k-1}^{2}}{\phi_{k-1}}+\sum_{i=1}^{k}\left(\frac{\mu}{\theta_{i}^{(k)}}\right)^{2}\frac{\left(s_{k,i}^{(k)}\right)^{2}}{\theta_{i}^{(k)}-\mu}.\label{eq:omega-1}
\end{equation}
\end{lemma}

\begin{proof}
Since $\left\Vert \mu T_{k}^{-1}\right\Vert <1$, we obtain using a Neumann series
\begin{eqnarray*}
\left(T_{k}-\mu I\right)^{-1}=\left(I-\mu T_{k}^{-1}\right)^{-1}T_{k}^{-1} & = & \left(\sum_{j=0}^{\infty}\mu^{j}T_{k}^{-j}\right)T_{k}^{-1}
\end{eqnarray*}
so that
\[
e_{k}^{T}\left(T_{k}-\mu I\right)^{-1}e_{k}=e_{k}^{T}T_{k}^{-1}e_{k}+\mu e_{k}^{T}T_{k}^{-2}e_{k}+\sum_{j=2}^{\infty}\mu^{j}e_{k}^{T}T_{k}^{-(j+1)}e_{k}.
\]
We now express the terms on the right-hand side using the CG coefficients
and the quantities from the spectral factorization of $T_{k}$. Using
$T_{k}=L_{k}D_{k}L_{k}^{T}$ we obtain $e_{k}^{T}T_{k}^{-1}e_{k}=\gamma_{k-1}$.
After some algebraic manipulation, see, e.g., \cite[p.~1369]{Me2020}
we get
\[
T_{k}^{-1}e_{k}=\gamma_{k-1}\Vert r_{k-1}\Vert\left[\begin{array}{c}
\frac{(-1)^{k-1}}{\Vert r_{0}\Vert}\\
\vdots\\
\frac{1}{\Vert r_{k-1}\Vert}
\end{array}\right]
\]
so that
\[
e_{k}^{T}T_{k}^{-2}e_{k}=e_{k}^{T}T_{k}^{-1}T_{k}^{-1}e_{k}=\gamma_{k-1}^{2}\sum_{i=0}^{k-1}\frac{\Vert r_{k-1}\Vert^{2}}{\Vert r_{i}\Vert^{2}}=\gamma_{k-1}^{2}\frac{\|p_{k-1}\|^{2}}{\|r_{k-1}\|^{2}}=\frac{\gamma_{k-1}^{2}}{\phi_{k-1}}.
\]
Finally,
\begin{eqnarray*}
e_{k}^{T}\left(\sum_{j=2}^{\infty}\mu^{j}T_{k}^{-(j+1)}\right)e_{k} & = & e_{k}^{T}S_{k}\left(\sum_{j=2}^{\infty}\mu^{j}\Theta_{k}^{-(j+1)}\right)S_{k}^{T}e_{k}
\end{eqnarray*}
where the diagonal entries of the diagonal matrix
\[
\sum_{j=2}^{\infty}\mu^{j}\Theta_{k}^{-(j+1)}
\]
have the form
\begin{eqnarray*}
\frac{1}{\theta_{i}^{(k)}}\left(\frac{\mu}{\theta_{i}^{(k)}}\right)^{2}\sum_{j=0}^{\infty}\left(\frac{\mu}{\theta_{i}^{(k)}}\right)^{j} & = & \frac{1}{\theta_{i}^{(k)}}\left(\frac{\mu}{\theta_{i}^{(k)}}\right)^{2}\frac{1}{1-\frac{\mu}{\theta_{i}^{(k)}}}=\left(\frac{\mu}{\theta_{i}^{(k)}}\right)^{2}\frac{1}{\theta_{i}^{(k)}-\mu}.
\end{eqnarray*}
Hence,
\[
e_{k}^{T}\left(\sum_{j=2}^{\infty}\mu^{j}T_{k}^{-(j+1)}\right)e_{k}=\sum_{i=1}^{k}\left(\frac{\mu}{\theta_{i}^{(k)}}\right)^{2}\frac{\left(s_{k,i}^{(k)}\right)^{2}}{\theta_{i}^{(k)}-\mu}.
\]
\end{proof}


Based on the previous lemma we can now express the coefficient $\gamma_{k}^{(\mu)}$.


\begin{theorem}
\label{lem:gammamu}Let $0<\mu<\theta_{1}^{(k)}$. Then it holds that
\begin{equation}
\left(\gamma_{k}^{(\mu)}\right)^{-1}=\frac{\mu}{\phi_{k}}+\sum_{i=1}^{k}\left(\frac{\mu}{\theta_{i}^{(k)}}\right)^{2}\eta_{i,k}^{(\mu)}.\label{eq:omega}
\end{equation}
\end{theorem}

\begin{proof}
We start with \eqref{eq:gammatilde}. Using the previous lemma
\begin{eqnarray*}
\left(\gamma_{k}^{(\mu)}\right)^{-1} & = & \mu+\mu\beta_{k}^{2}\gamma_{k-1}^{2}\phi_{k-1}^{-1}+\beta_{k}^{2}e_{k}^{T}\left(\sum_{j=2}^{\infty}\mu^{j}T_{k}^{-(j+1)}\right)e_{k}\\
 & = & \mu\left(1+\delta_{k}\phi_{k-1}^{-1}\right)+\beta_{k}^{2}\sum_{i=1}^{k}\left(\frac{\mu}{\theta_{i}^{(k)}}\right)^{2}\frac{\left(s_{k,i}^{(k)}\right)^{2}}{\theta_{i}^{(k)}-\mu}\\
 & = & \mu\phi_{k}^{-1}+\sum_{i=1}^{k}\left(\frac{\mu}{\theta_{i}^{(k)}}\right)^{2}\frac{\left(\beta_{k}s_{k,i}^{(k)}\right)^{2}}{\theta_{i}^{(k)}-\mu},
\end{eqnarray*}
where we have used relation \eqref{eq:phiupdate}.
\end{proof}

Obviously, using \eqref{eq:omega}, the basic Gauss-Radau upper bound
\eqref{eq:GR} and the simple upper bound in \eqref{eq:basic} are
close to each other if and only if
\begin{equation}
\sum_{i=1}^{k}\left(\frac{\mu}{\theta_{i}^{(k)}}\right)^{2}\eta_{i,k}^{(\mu)}\,\ll\,\frac{\mu}{\phi_{k}},\label{eq:negligible}
\end{equation}
which can also be written as
\begin{equation}
\left(\frac{\mu}{\theta_{1}^{(k)}}\right)^{2}\frac{\eta_{1,k}^{(\mu)}}{\mu}+\sum_{i=2}^{k}\left(\frac{\beta_{k}s_{k,i}^{(k)}}{\theta_{i}^{(k)}}\right)^{2}\frac{\mu}{\theta_{i}^{(k)}-\mu}\,\ll\,\phi_{k}^{-1}.\label{eq:negligible2}
\end{equation}

{Under the assumptions formulated in Section~\ref{sec:assumptions},
in particular that}
$\lambda_{1}$ is well separated from $\lambda_{2}$,
and that $\mu$ is a tight underestimate to $\lambda_{1}$ in the
sense of \eqref{eq:separate}, the sum of terms on the left-hand side
of \eqref{eq:negligible2} can be replaced by its tight upper bound
\begin{equation}
\left(\frac{\lambda_{1}}{\theta_{1}^{(k)}}\right)^{2}\frac{\eta_{1,k}^{(\mu)}}{\mu}+\sum_{i=2}^{k}\left(\frac{\beta_{k}s_{k,i}^{(k)}}{\theta_{i}^{(k)}}\right)^{2}\frac{\lambda_{1}}{\theta_{i}^{(k)}-\lambda_{1}}\label{eq:upper2}
\end{equation}
which simplifies the explanation of the dependence of the sum in \eqref{eq:negligible2}
on $\mu$.

The second term in \eqref{eq:upper2} is independent of
$\mu$ and its size depends only on the behaviour of the underlying
Lanczos process. Here
\begin{equation}
\left(\frac{\beta_{k}s_{k,i}^{(k)}}{\theta_{i}^{(k)}}\right)^{2}=\frac{\left\Vert A\left(V_{k}s_{:,i}^{(k)}\right)-\theta_{i}^{(k)}\left(V_{k}s_{:,i}^{(k)}\right)\right\Vert ^{2}}{\left(\theta_{i}^{(k)}\right)^{2}}\label{eq:term1}
\end{equation}
can be seen as the relative accuracy to which the $i$th Ritz value
approximates an eigenvalue, and the size of the term
\begin{equation}
\frac{\lambda_{1}}{\theta_{i}^{(k)}-\lambda_{1}},\qquad i\geq2,\label{eq:term2}
\end{equation}
depends on the position of $\theta_{i}^{(k)}$ relatively to the smallest
eigenvalue. In particular, one can expect that the term \eqref{eq:term2}
can be of size $\mathcal{O}(1)$ if $\theta_{i}^{(k)}$ approximates
smallest eigenvalues, and it is small if $\theta_{i}^{(k)}$ approximates
largest eigenvalues.

Using the previous simplifications and assuming phase~2, the basic
Gauss-Radau upper bound \eqref{eq:GR} and the rightmost upper bound
in \eqref{eq:basic} are close to each other if and only if
\begin{equation}
\frac{\eta_{1,k}^{(\mu)}}{\mu}+\sum_{i=2}^{k}\left(\frac{\beta_{k}s_{k,i}^{(k)}}{\theta_{i}^{(k)}}\right)^{2}\frac{\lambda_{1}}{\theta_{i}^{(k)}-\lambda_{1}}\,\ll\,\phi_{k}^{-1}.\label{eq:negligible3}
\end{equation}
From {Section~\ref{sec:alphamu}} we know that $\eta_{1,k}^{(\mu)}$
goes to zero in phase~2. Hence, if
\begin{equation}
\eta_{1,k}^{(\mu)}<\mu,\label{eq:lessmu}
\end{equation}
which will happen for $k$ sufficiently large, then the first term
in \eqref{eq:negligible3} is smaller than the term on the right-hand
side.

As already mentioned, the sum of positive terms in
\eqref{eq:negligible3} depends only on approximation properties of
the underlying Lanczos process, that are not easy to predict in general.
Inspired by our model problem described in Section~\ref{sec:model},
we can just give an intuitive explanation why the sum could be small
in phase~2.

Phase~2 occurs in later CG iterations and it is
related to the convergence of the smallest Ritz value to the smallest
eigenvalue. If the smallest eigenvalue is well approximated by the
smallest Ritz value (to a high relative accuracy), then one can expect
that many eigenvalues of $A$ are relatively well approximated by
 Ritz values. If the eigenvalue $\lambda_{j}$ of
$A$ is well separated from the other eigenvalues and if it is well
approximated by a Ritz value, then the corresponding term \eqref{eq:term1}
measuring the relative accuracy to which $\lambda_{j}$ is approximated,
is going to be small.

\begin{figure}[!htbp]
\centering{}\includegraphics[width=11cm]{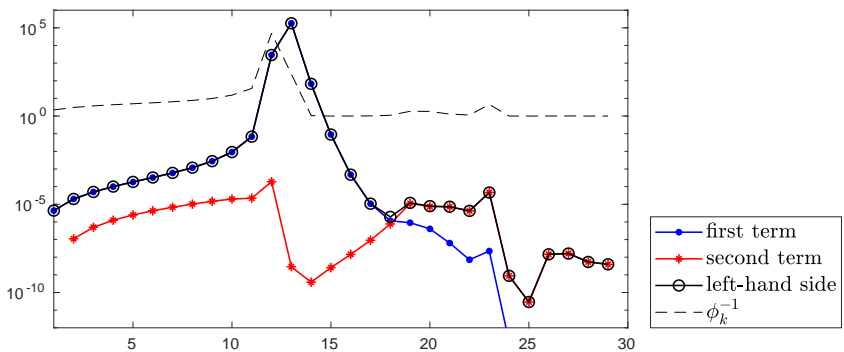} \caption{{The first and second term in \eqref{eq:upper2}, left-hand side of~\eqref{eq:negligible2}, and $\phi_k^{-1}$.}}
\label{fig:1-1}
\end{figure}

In particular, in our model problem, the smallest eigenvalues
are well separated from each other, and in phase~2 they are well
approximated by Ritz values. Therefore, the corresponding terms \eqref{eq:term1}
are small. Hence, the Ritz values that did not converge yet in phase~2,
are going to approximate eigenvalues in clusters which do not correspond
to smallest eigenvalues, i.e., for which the terms \eqref{eq:term2}
are small; see also Figure~\ref{fig-4} and Figure~\ref{fig-7}. In
our model problem, the sum of positive terms in \eqref{eq:negligible3}
is small in phase~2 because either \eqref{eq:term1} or \eqref{eq:term2}
are small. Therefore, one can expect that the validity of \eqref{eq:negligible3}
will mainly depend on the size of the first term in \eqref{eq:negligible3};
see Figure~\ref{fig:1-1}.

The size of the sum of positive terms in \eqref{eq:negligible3}
obviously depends on the clustering and the distribution of the eigenvalues,
and we cannot guarantee in general that it will be small in phase~2.
For example, it need not be small if the smallest eigenvalues of $A$
are clustered.

\section{Detection of phase~2}
\label{sec:detection}
For our model problem it is not hard to detect phase~2 from the coefficients
that are available during the computations. We first observe, see
Figure~\ref{fig:1-1}, that the coefficients
\begin{equation}
\gamma_{k}^{(\mu)}\quad\mbox{and}\quad\frac{\phi_{k}}{\mu},\label{eq:bracketing}
\end{equation}
and the corresponding bounds \eqref{eq:GR} and \eqref{eq:basic}
visually coincide from the beginning up to some iteration $\ell_{1}$.
From iteration $\ell_{1}+1$, the Gauss-Radau upper bound \eqref{eq:GR}
starts to be a much better approximation to the squared $A$-norm
of the error than the simple upper bound \eqref{eq:basic}. When phase~2
occurs, the Gauss-Radau upper bound \eqref{eq:GR} loses its accuracy
 and, starting from iteration $\ell_{2}$ (approximately
when \eqref{eq:lessmu} holds), it will again visually coincide with
the simple upper bound \eqref{eq:basic}. We observe that phase~2
occurs at some iteration $k$ where the two coefficients \eqref{eq:bracketing}
significantly differ, i.e., for $\ell_{1}<k<\ell_{2}.$ To measure
the agreement between the coefficients \eqref{eq:bracketing}, we can
use the easily computable relative distance
\begin{equation}
\frac{\frac{\phi_{k}}{\mu}-\gamma_{k}^{(\mu)}}{\gamma_{k}^{(\mu)}}=\phi_{k}\left[\left(\frac{\mu}{\theta_{1}^{(k)}}\right)^{2}
\frac{\eta_{1,k}^{(\mu)}}{\mu}+\sum_{i=2}^{k}\left(\frac{\beta_{k}s_{k,i}^{(k)}}{\theta_{i}^{(k)}}\right)^{2}\frac{\mu}{\theta_{i}^{(k)}
-\mu}\right].\label{eq:crit1}
\end{equation}
We will consider this relative distance to be small, if it is smaller
than 0.5.

\begin{figure}[!htbp]
\centering{}\includegraphics[width=11cm]{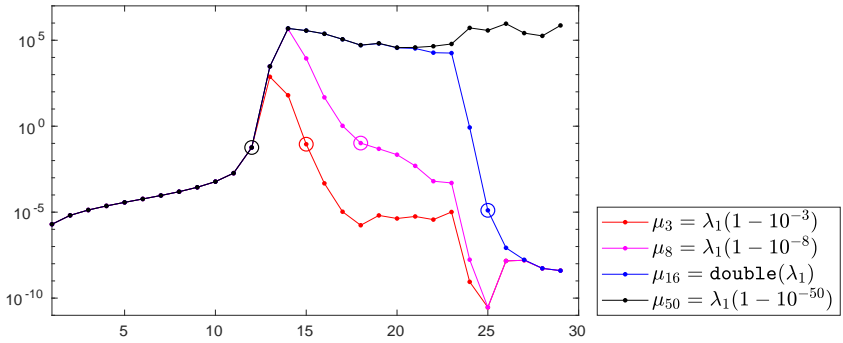} \caption{The behaviour of the relative distance in \eqref{eq:crit1} for various values
of $\mu$.}
\label{fig:71}
\end{figure}

The behavior of the term in \eqref{eq:crit1} for various values of
$\mu$ is shown in Figure~\ref{fig:71}. The index $\ell_{1}=12$ is
the same for all considered values of $\mu$. For $\mu_{3}$ we get
$\ell_{2}=15$ (red circle), for $\mu_{8}$ we get $\ell_{2}=18$
(magenta circle), for $\mu_{16}$ $\ell_{2}=25$ (blue circle),
and finally, for $\mu_{50}$ there is no index $\ell_{2}$.

As explained in the previous section, in more complicated
cases we cannot guarantee in general a similar behaviour of the relative
distance \eqref{eq:crit1} as in our model problem. For example, in
{many} practical problems we sometimes observe {a} staircase behaviour of the
$A$-norm of the error, when few iterations of stagnation are followed
by few iterations of rapid convergence. In such cases, the quantity
\eqref{eq:crit1} can oscillate several times and it
{can be impossible} to use it for detecting phase~2.
{Therefore, in general, we are not able to detect the beginning of phase~2 using \eqref{eq:crit1} reliably.
Nevertheless, in particular cases, the formulas \eqref{eq:omega} and
\eqref{eq:crit1} can be helpful.}

\section{{Upper bounds with a guaranteed accuracy}}
\label{sec:improved}
{
In some applications it might be of interest to obtain upper bounds
on the $A$-norm of the error that are sufficiently accurate. From
the previous sections we know that the basic Gauss-Radau upper bound
at iteration $k$ can be delayed, and, therefore, it can overestimate
the quantity of interest significantly. Nevertheless, going back
in the convergence history, we can easily find an iteration
index $\ell\leq k$ such that for all $0\leq i \leq \ell$, the
sufficiently accurate upper bound can be found. To find such $\ell$,
we will use the ideas described in \cite{StTi2002} and \cite{MePaTi2021}.
}

For integers {$k\geq j \geq \ell\geq 0$}, let us denote
\[
\Delta_{j}=\gamma_{j}\left\Vert r_{j}\right\Vert ^{2},\quad\Delta_{\ell:k}=\sum_{j=\ell}^{k}\Delta_{j},\quad\mbox{and}\quad\Delta_{j:j-1}=0.
\]
Denoting $\varepsilon_{j} \equiv \| x - x_j \|_A^2$, the relation
\eqref{eq:delay} takes the form
\begin{equation}
\varepsilon_{\ell}=\Delta_{\ell:k-1}+\varepsilon_{k},\label{eq:HSid}
\end{equation}
A more accurate bound at iteration $\ell$ is obtained such that the last
term in \eqref{eq:HSid} is replaced by the basic lower or upper bounds
on $\varepsilon_{k}$. In particular, the improved Gauss-Radau
upper bound {at iteration $\ell$} can be defined as
\begin{equation}
\Omega{}_{\ell:k}^{(\mu)}\,=\,\Delta_{\ell:k-1}+\gamma_{k}^{(\mu)}\left\Vert r_{k}\right\Vert ^{2},\label{eq:improved}
\end{equation}
and the improved Gauss lower bound is given by $\Delta_{\ell:k}$.

To guarantee the relative accuracy of the improved Gauss-Radau upper bound,
we would like to find the largest iteration index $\ell\leq k$ in the convergence history such that
\begin{equation}
\frac{\Omega{}_{\ell:k}^{(\mu)}-\varepsilon_{\ell}}{\varepsilon_{\ell}}\leq\tau\label{eq:crit1-1}
\end{equation}
where $\tau$ is a prescribed tolerance, say, $\tau=0.25$. Since
\[
\frac{\text{\ensuremath{\Omega{}_{\ell:k}^{(\mu)}}}-\varepsilon_{\ell}}{\varepsilon_{\ell}}<\frac{\Omega{}_{\ell:k}^{(\mu)}-\Delta_{\ell:k}}{\Delta_{\ell:k}}=\frac{\left\Vert r_{k}\right\Vert ^{2}\left(\gamma_{k}^{(\mu)}-\gamma_{k}\right)}{\Delta_{\ell:k}},
\]
we can require $\ell \leq k$ to be the {largest} integer such that
\begin{equation}
\frac{\left\Vert r_{k}\right\Vert ^{2}\left(\gamma_{k}^{(\mu)}-\gamma_{k}\right)}{\Delta_{\ell:k}}\leq\tau.\label{eq:crit2}
\end{equation}
If \eqref{eq:crit2} holds, then also \eqref{eq:crit1-1} holds. The
just described adaptive strategy for obtaining $\ell$ giving a sufficiently
accurate upper bound is summarized in Algorithm~\ref{alg:pseudo}.

\begin{algorithm}
\caption{CG with the improved Gauss-Radau upper bound}
\label{alg:pseudo}

\begin{algorithmic}[0]

\State \textbf{input} $A$, $b$, $x_{0}$, $\mu$, $\tau$

\State $r_{0}=b-Ax_{0}$, $p_{0}=r_{0}$

\State $\ell=0$, $\gamma_{0}^{(\mu)}=\frac{1}{\mu}$

\For{{$k=0,\dots,$}}

\State \texttt{cgiter($k$)}

\While{$k\geq \ell$ and \eqref{eq:crit2}}

\State accept $\Omega_{\ell:k}^{(\mu)}$

\State $\ell=\ell+1$

\EndWhile

\State $\gamma_{k+1}^{(\mu)}=\frac{\gamma_{k}^{(\mu)}-\gamma_{k}}{\mu(\gamma_{k}^{(\mu)}-\gamma_{k})+\delta_{k+1}}$

\EndFor

\end{algorithmic}
\end{algorithm}

Note that
\[
\frac{\Omega{}_{\ell:k}^{(\mu)}-\varepsilon_{\ell}}{\varepsilon_{\ell}}+\frac{\varepsilon_{\ell}-\Delta_{\ell:k}}{\varepsilon_{\ell}}<\frac{\Omega{}_{\ell:k}^{(\mu)}-\Delta_{\ell:k}}{\Delta_{\ell:k}},
\]
i.e., if \eqref{eq:crit2} holds, then $\tau$ represents also an
upper bound on the sum of relative errors of the improved lower and
upper bounds. In other words, if $\ell$ is such that \eqref{eq:crit2}
is satisfied, then both the improved Gauss-Radau upper bound as well
as the improved Gauss lower bound are sufficiently accurate.
{For a heuristic strategy}
focused on improving the accuracy of the Gauss lower bound, see \cite{MePaTi2021}.

In the previous sections we have seen that the basic Gauss-Radau upper bound
is delayed, in particular in phase~2. The delay of the basic Gauss-Radau
upper bound can be defined as the smallest nonnegative integer $j$
such that
\begin{equation}
\gamma_{\ell+j+1}^{(\mu)}\left\Vert r_{\ell+j+1}\right\Vert ^{2}<\varepsilon_{\ell}.\label{eq:udelay}
\end{equation}
Having sufficiently accurate lower and upper bounds (e.g., if \eqref{eq:crit2}
is satisfied), we can approximately determine the delay of the basic
Gauss-Radau upper bound as the smallest $j$ satisfying \eqref{eq:udelay}
where $\varepsilon_{\ell}$ in \eqref{eq:udelay} is replaced by its
tight lower bound $\Delta_{\ell:k}$ .

\section{Conclusions}
\label{sec:conclusions}
In this paper we discussed and analyzed the behaviour
of the Gauss-Radau upper bound on the $A$-norm of the error in CG.
In particular, we concentrated on the phenomenon observed during computations
showing that, in later CG iterations, the upper bound loses its accuracy,
it is almost independent of $\mu$, and visually coincides with the
simple upper bound. We explained that this phenomenon is closely related
to the convergence of the smallest Ritz value to the smallest eigenvalue
of $A$. It occurs when the smallest Ritz value is a better approximation
to the smallest eigenvalue than the prescribed underestimate $\mu$.
We developed formulas that can be helpful in understanding {this behavior.
Note that} the loss of accuracy of
the Gauss-Radau upper bound is not directly linked to rounding errors
in computations of the bound, but it is related to the finite precision
behaviour of the underlying Lanczos process. In more detail, the phenomenon
can occur when solving linear systems with clustered eigenvalues.
However, the results of finite precision CG computations can be seen
(up to some small inaccuracies) as the results of the exact CG algorithm
applied to a larger system with the system matrix having clustered
eigenvalues. Therefore, one can expect that the discussed phenomenon
can occur in practical computations not only when $A$ has clustered
eigenvalues, but also whenever orthogonality is lost in the CG algorithm.

\subsection*{Acknowledgments}
The work of Petr~Tich\'{y} was supported by the Grant Agency of the Czech Republic under grant no. 20-01074S.


\end{document}